\documentclass{amsart}

\usepackage{amsmath,amssymb,amsthm}

\usepackage{units} 

\usepackage{pinlabel}

\newtheorem{prop}{Proposition}[section]
\newtheorem{thm}[prop]{Theorem}
\newtheorem{lem}[prop]{Lemma}

\theoremstyle{definition}

\newtheorem{defn}[prop]{Definition}

\newtheorem{rem}[prop]{Remark}

\newtheorem*{ack}{Acknowledgements}

\newtheorem{quest}[prop]{Question}

\def\co{\colon\thinspace}

\newcommand{\alphast}{\alpha_{\mathrm{st}}}

\newcommand{\C}{\mathbb{C}}
\newcommand{\CP}{\mathbb{C}\mathrm{P}}

\newcommand{\rmd}{\mathrm{d}}

\newcommand{\rme}{\mathrm{e}}

\newcommand{\rmi}{\mathrm{i}}

\newcommand{\LL}{\mathcal{L}}

\newcommand{\omegast}{\omega_{\mathrm{st}}}
\newcommand{\calO}{\mathcal{O}}

\newcommand{\PD}{\mathrm{PD}}

\newcommand{\R}{\mathbb{R}}

\newcommand{\Rst}{R_{\mathrm{st}}}
\newcommand{\rot}{\mathtt{rot}}

\newcommand{\SO}{\mathrm{SO}}
\newcommand{\SU}{\mathrm{SU}}

\newcommand{\tb}{\mathtt{tb}}

\newcommand{\U}{\mathrm{U}}

\newcommand{\xist}{\xi_{\mathrm{st}}}

\newcommand{\Z}{\mathbb{Z}}
\newcommand{\oz}{\overline{z}}

\DeclareMathOperator{\Int}{Int}

\begin{document}

\author[H.~Geiges]{Hansj\"org Geiges}
\address{Mathematisches Institut, Universit\"at zu K\"oln,
Weyertal 86--90, 50931 K\"oln, Germany}
\email{geiges@math.uni-koeln.de}
\author[K.~Zehmisch]{Kai Zehmisch}
\address{Mathematisches Institut, WWU M\"unster,
Einstein\-stra\-{\ss}e 62, 48149 M\"unster, Germany}
\email{kai.zehmisch@uni-muenster.de}

\title[Odd-symplectic forms and minimality]{Odd-symplectic
forms via surgery and minimality in symplectic dynamics}

\date{}

\begin{abstract}
We construct an infinite family of odd-symplectic forms
(also known as Hamiltonian structures) on the $3$-sphere $S^3$ that do not
admit a symplectic cobordism to the standard contact structure
on~$S^3$. This answers in the negative a question raised by
Joel Fish motivated by the search for minimal characteristic
flows.
\end{abstract}



\thanks{The authors are partially supported by the SFB/TRR 191
`Symplectic Structures in Geometry, Algebra and Dynamics', funded by the
Deutsche Forschungsgemeinschaft.}

\maketitle


\section{Introduction}
The plugs for Hamiltonian flows constructed by
V.~Ginzburg~\cite{ginz95,ginz97} and M.~Herman~\cite{herm99,herm}
allow one to produce smooth Hamiltonian flows without periodic orbits
on compact hypersurfaces in $\R^{2n}$, $n\geq 3$. For an alternative
construction of Hamiltonian plugs and a guide to the more recent
literature on the subject see~\cite{grz16}.

The existence of aperiodic Hamiltonian flows prompted Herman at his
1998 ICM address~\cite{herm98} to raise the question whether one can find
compact, connected hypersurfaces in $\R^{2n}$, $n\geq 2$,
on which the characteristic flow is not only aperiodic, but
\emph{minimal}, that is, where every orbit is dense.
For $n=2$, this question has recently been answered in the negative by
J. Fish and H. Hofer, see~\cite{fish15}.

\begin{thm}[Fish--Hofer]
\label{thm:Fish-Hofer}
Let $H$ be a smooth and proper Hamiltonian on~$\R^4$. Then no
energy level of $H$ is minimal.
\end{thm}

In this context, Fish has posed a question concerning the
existence of certain symplectic cobordisms, very much in the
spirit of symplectic dynamics as defined in~\cite{brho12}.
Following~\cite{ginz92}, we call a closed $2$-form $\omega$
on a $(2n-1)$-dimensional manifold $M$ an \emph{odd-symplectic form}
if $\omega$ is of rank~$2n-2$, that is, if it has a $1$-dimensional
kernel. The terminology `Hamiltonian structure' is also in use,
see~\cite{civo15}. We shall always assume $M$ to be oriented;
equivalently, the characteristic line bundle $\ker\omega$
is oriented.

In the following question, the $3$-sphere $S^3$ is given its
standard orientation as the unit sphere in~$\R^4$. The contact structure
$\xi$ on $S^3$ is assumed to be positive, that is, $\alpha\wedge\rmd\alpha$
is a positive volume form for any choice of contact form $\alpha$
with $\ker\alpha=\xi$. Moreover, for condition (ii) below
it is assumed that a coorientation and hence orientation for
$\xi$ has been chosen by fixing $\alpha$ up to multiplication by
a positive function.  The symplectic $4$-manifold $(W,\Omega)$
is oriented by the volume form~$\Omega\wedge\Omega$, and
the boundary $\partial W$ of $W$ is given the induced orientation.

\begin{quest}[Fish]
\label{quest:Fish}
Let $\omega$ be an odd-symplectic form on the $3$-sphere $S^3=:S^3_-$.
Can one find a contact structure $\xi$ on $S^3=:S^3_+$ and a
compact symplectic $4$-manifold $(W,\Omega)$ such that
$\partial W=S^3_+-S^3_-$ and
\begin{itemize}
\item[(i)] $\Omega|_{TS^3_-}=\omega$,
\item[(ii)] $\Omega|_{\xi}>0$?
\end{itemize}
\end{quest}

\begin{rem}
With the help of open books or by
contact surgeries one can build a symplectic cobordism from
$(S^3_+,\xi)$ to $(S^3,\xist)$, i.e.\ the $3$-sphere with
its standard contact structure, see~\cite[Theorem~1.2]{etho02},
\cite[Theorem~3.4]{geze13} ---
this statement hinges on the fact that, as a consequence of Eliashberg's
classification of contact structures on the $3$-sphere,
$\xi$ is either diffeomorphic
to~$\xist$, or it is an overtwisted contact structure.
By concatenating this with a purported cobordism from $(S^3_-,\omega)$
to $(S^3_+,\xi)$, one would obtain one from $(S^3_-,\omega)$
to $(S^3,\xist)$. Thus, we may assume that $(S^3_+,\xi)=(S^3,\xist)$
in Fish's question.
\end{rem}

Here is a brief explanation of the relevance of Question~\ref{quest:Fish}.
Given a compact, connected hypersurface $M$
in standard symplectic space $(\R^{2n},\omegast)$,
the pull-back $\omega:=\omegast|_{TM}$ of
the symplectic form $\omegast$ to~$M$ is odd-symplectic.
The \emph{characteristics} of $(M,\omega)$
are the integral curves of the line field on $M$ defined by $\ker\omega$.
If $M$ is written as the regular level set of a smooth
function $H\co\R^{2n}\rightarrow\R$, the characteristics coincide
with the flow lines of the Hamiltonian vector field $X_H$
defined by $\omegast(X_H,\,.\,)=-\rmd H$, no matter what choice of~$H$.
Thus, up to parametrisation, the Hamiltonian dynamics is
encoded in the odd-symplectic form on the hypersurface.

For a level set $H^{-1}(c)$ as in Theorem~\ref{thm:Fish-Hofer},
a cobordism as sought for in Question~\ref{quest:Fish}
exists: simply enclose $H^{-1}(c)$ inside a large sphere and take
$(W,\Omega)$ to be the part of $(\R^4,\omegast)$ bounded by these
two hypersurfaces.

In order to prove Theorem~\ref{thm:Fish-Hofer}, Fish and Hofer
compactify $(\R^4,\omegast)$ to the complex projective plane
$\CP^2$ with its Fubini--Study symplectic form. They then use
results of Gromov on the existence of pseudoholomorphic spheres
in $\CP^2$ to find a nontrivial minimal set of characteristics
in $H^{-1}(c)$ as the limit set of ends of
what they call feral pseudoholomorphic curves.

For a general odd-symplectic manifold $(M^3,\omega)$, having
a cobordism as in Question~\ref{quest:Fish} might allow one to carry over 
parts of this argument. That such a strategy might be viable
is exemplified by Hofer's proof~\cite{hofe93}
of the Weinstein conjecture for overtwisted contact structures.
As explained in \cite[Corollary~3.5]{geze13}, in this
situation the existence of a periodic Reeb orbit (so in particular
the non-minimality of the Reeb flow) follows
from a study of pseudoholomorphic curves in a
symplectic cobordism from the
given overtwisted contact $3$-manifold to $(S^3,\xist)$.

Alas, the main result of this note says that
the answer to Question~\ref{quest:Fish} is negative, in general.

\begin{thm}
\label{thm:main}
There is an infinite family of odd-symplectic forms on $S^3$,
distinguished by a homotopical invariant, for which a cobordism
as in Question~\ref{quest:Fish} does not exist.
\end{thm}

We describe an explicit construction of this infinite family of
examples; the homotopical invariant in question and
how to compute it will be explained in the process. We shall
also discuss the dynamics of these odd-symplectic forms in greater detail.

An earlier result in this direction is due to K.~Cieliebak and
E.~Volkov, see~\cite[Corollary~6.21]{civo15}. They exhibit
an example of an odd-symplectic form $\omega$ on $S^3$
for which there is no topologically trivial symplectic cobordism
to an odd-symplectic form on $S^3$ whose characteristics
are given by the Reeb flow of the standard contact form.
The latter condition is not actually more restrictive than what is
required by Question~\ref{quest:Fish}, as a construction by
Eliashberg~\cite{elia04} allows one to modify
the cobordism accordingly, cf.\ step (iii) in the proof of
\cite[Theorem~6]{geig06}. Topological triviality of the
cobordism, however, is not assumed in our Theorem~\ref{thm:main}.

The example of Cieliebak and Volkov arises as the boundary of an
exotic symplectic ball (containing an exact Lagrangian $2$-torus),
and the non-existence of the desired
cobordism is shown by appealing to Gromov's uniqueness theorem
for symplectic structures on $\R^4$ standard at infinity.
In other words, their example relies on two deep results
of $4$-dimensional symplectic topology. Our infinite list of
examples, by contrast, is constructed by an `elementary'
surgical procedure, and the non-existence of a symplectic cobordism
with the described properties follows from McDuff's result~\cite{mcdu91},
cf.~\cite[Corollary~3.2]{geze13},
which says that if $(S^3,\xist)$ arises as one boundary component
of a compact symplectic $4$-manifold with weakly convex boundary,
then the boundary is in fact connected.
\section{Idea of the construction}
We start with the product manifold $S^1\times S^2$ and the obvious
odd-symplectic form $\omega_{S^2}$ obtained by pulling back
the standard area form from $S^2$ (of area~$4\pi$). In a neighbourhood of
a circle $S^1\times\{*\}$ we can realise $\omega_{S^2}$
as an exact form $\rmd\alpha$ for some contact form~$\alpha$.
We then perform contact surgery, see \cite[Chapter~6]{geig08},
inside this neighbourhood along a Legendrian knot that is
topologically isotopic to $S^1\times\{*\}$. This surgery,
interpreted as a handle attachment, produces
a symplectic cobordism (in a sense that we shall specify)
from $(S^1\times S^2,\omega_{S^2})$ to the surgered
odd-symplectic manifold.

In Section~\ref{section:surgery} we are going to describe this
in detail and show that the surgered manifold is a $3$-sphere.
The odd-symplectic forms $\omega$ on $S^3$ produced in this way
(by choosing a Legendrian knot from an infinite family) can be
distinguished by the Hopf invariant of the oriented line
field $\ker\omega\subset TS^3$. We use a version of this
invariant due to R.~Gompf~\cite{gomp98}, which can be computed from
an almost complex structure $J=J_{\omega}$ on a filling $W_{\omega}$
of $(S^3,\omega)$, that is, a compact manifold with boundary
$\partial W_{\omega}=S^3$, and with $\omega>0$ on the
$J$-invariant tangent $2$-plane field $TS^3\cap J(TS^3)$.
This so-called $d_3$-invariant, which has the advantage not to depend
on a choice of trivialisation of the tangent bundle $TS^3$,
will be described in Section~\ref{section:d3}, and in
Section~\ref{section:d3-examples} we compute it for our
examples.

The non-existence of a symplectic cobordism (as
in Question~\ref{quest:Fish}) from any of these $(S^3,\omega)$
to $(S^3,\xist)$ will be shown in Section~\ref{section:proof}.
Assuming there were such a cobordism, we could concatenate it
with the surgery cobordism to obtain a cobordism from
$(S^1\times S^2,\omega_{S^2})$ to $(S^3,\xist)$. A result
from \cite{geze17} about symplectic cobordisms between symplectic
fibrations would allows us to connect two copies of this
cobordism by a symplectic cobordism between the
two boundary components $(S^1\times S^2,\omega_{S^2})$,
or better: by a symplectic cobordism from the empty set to
these two manifolds.
The resulting manifold would be a connected weak symplectic filling
of the disjoint union of two copies of $(S^3,\xist)$,
contradicting a result of McDuff~\cite{mcdu91}.

Finally, in Section~\ref{section:dynamics} we discuss the
dynamics of these examples. After a slight modification of
the odd-symplectic form, which does not affect the property
of them not being symplectically cobordant to
$(S^3,\xist)$, the minimality or otherwise of these examples
is an open problem.

\begin{rem}
According to a conjecture due to W.~Gottschalk,
there are no minimal flows on $S^3$ whatsoever. From this
perspective it is worth noting that our construction allows one
to produce odd-symplectic manifolds not diffeomorphic to $S^3$ 
for which a symplectic cobordism to $(S^3,\xist)$ as in
Question~\ref{quest:Fish} does not
exist. For instance, one can replace $S^1\times S^2$
in the construction by $S^1\times\Sigma_g$, with $\Sigma_g$ the
orientable surface of genus~$g$.

There are very few examples of minimal flows on closed $3$-manifolds:
irrational flows on the $3$-torus and on nilmanifolds,
and the horocycle flow on the unit tangent bundle of
compact hyperbolic surfaces, see~\cite{ahm63}, \cite[Section~2.2]{ghys92},
\cite[Section~4]{mark71}.
\end{rem}
\section{Surgical description of the examples}
\label{section:surgery}
We begin with the topological aspect of the surgery.

\begin{lem}
\label{lem:surgery-top}
Surgery along $S^1\times\{*\}\subset S^1\times S^2$ with integral framing
produces~$S^3$.
\end{lem}

\begin{proof}
We split $S^1\times S^2$ into two solid tori $V_0,V_1$, with $V_0$
a tubular neighbourhood of $S^1\times\{*\}$. These solid tori
carry canonical longitudes $\lambda_0,\lambda_1$, respectively,
given by (the class of) a fibre in the fibration
$S^1\times S^2\rightarrow S^2$. The respective meridians will be denoted by
$\mu_0,\mu_1$. The gluing of $V_0,V_1$ that gives $S^1\times S^2$
is described by the identification $\mu_0=-\mu_1$, $\lambda_0=\lambda_1$.

Surgery along $S^1\times\{*\}$ with integral framing means that we
cut out $V_0$ and reglue a solid torus $S^1\times D^2$ by sending
its meridian $\mu:=\{*\}\times\partial D^2$ to
\[ k\mu_0+\lambda_0 = -k\mu_1+\lambda_1\]
for some $k\in\Z$. Since there is a diffeomorphism of
$V_1$ that sends $-k\mu_1+\lambda_1$ to $\lambda_1$ (a $k$-fold
right-handed Dehn twist along a meridional disc), this gluing
amounts to the same as sending $\mu$ to~$\lambda_1$,
which produces~$S^3$.
\end{proof}

Identify the solid torus $V_0$ with $S^1\times D^2\subset S^1\times S^2$,
and write $\theta$ for the $S^1$-coordinate. Recall that
the Reeb vector field $R$ of a contact form $\alpha$ is defined
by the conditions $\rmd\alpha(R,\,.\,)=0$ and $\alpha(R)=1$.
We now want to choose a contact form $\alpha_0$ on $V_0$ with Reeb vector
field $R_0=\partial_{\theta}$ and $\rmd\alpha_0=\omega_{S^2}$.
Of course, we could simply set $\alpha_0=\rmd\theta+(x\,\rmd y-y\,\rmd x)/2$,
with $x,y$ cartesian coordinates on~$D^2$. For the computation of
the classical invariants of Legendrian knots inside the contact
manifold $(V_0,\ker\alpha_0)$, however, it is more convenient
to define $\alpha_0$ by an identification of $V_0$ with a
subset of $(S^3,\alphast)$, where
\begin{eqnarray*}
\alphast & := & \frac{\rmi}{2}(z_1\,\rmd\oz_1-\oz_1\,\rmd z_1+
                   z_2\,\rmd\oz_2-\oz_2\,\rmd z_2)\\
         & =  & r_1^2\,\rmd\varphi_1+ r_2^2\,\rmd\varphi_2\\
         & =  & x_1\,\rmd y_1-y_1\,\rmd x_1+x_2\,\rmd y_2-y_2\,\rmd x_2
\end{eqnarray*}
is the standard contact form on $S^3\subset\C^2$
with coordinates $z_j=r_j\rme^{\rmi\varphi_j}=x_j+\rmi y_j$, $j=1,2$.
This is the contact form that defines the standard contact
structure $\xist=\ker\alphast$. Its Reeb vector field
is $\Rst=\partial_{\varphi_1}+\partial_{\varphi_2}$.

The vector field $\Rst$ defines the Hopf fibration
\[ \begin{array}{ccc}
\C^2\supset S^3 & \longrightarrow & S^2=\CP^1\\
(z_1,z_2)       & \longmapsto     & [z_1:z_2].
\end{array} \]
The $1$-form $\alphast$ is a connection form on this
principal $S^1$-bundle and $\rmd\alphast$ is the
pull-back of a symplectic form on $S^2$ of total area~$2\pi$,
cf.\ the discussion of the Boothby--Wang construction
in~\cite[Section~7.2]{geig08}.

Since the total area of the area form $\omega_{S^2}$ on $S^2$
is $4\pi$, we may identify $V_0=S^1\times D^2$ with the solid torus
\[ V_{S^3}:=\bigl\{ (z_1,z_2)\in S^3\co |z_2|\leq\nicefrac{3}{4}\bigr\}\]
in $S^3$ such that
\begin{itemize}
\item[(i)] $S^1\times\{0\}$ is sent to the Reeb orbit
\[\bigl\{ (\rme^{\rmi\theta},0)\in S^3\co\theta\in S^1=\R/2\pi\Z\bigr\} \]
of $\alphast$,
\item[(ii)] $\{0\}\times D^2\subset S^1\times D^2$ with area form
$\omega_{S^2}$ is sent area-preservingly to a transverse disc to
the flow of $\Rst$ with area form $\rmd\alphast$, and
\item[(iii)] the flow lines of $\partial_{\theta}$ are sent to
those of $\Rst$ (with the parametrisations given by these
respective vector fields).
\end{itemize}
Then define the contact form $\alpha_0$ on $V_0$ as the
pull-back of $\alphast$ under this identification.
By construction, the vector field $\partial_{\theta}$ is
the Reeb vector field of~$\alpha_0$, so $\rmd\alpha_0$ is invariant under the
flow of~$\partial_{\theta}$, and by (ii) we then have
$\rmd\alpha_0=\omega_{S^2}$ on~$V_0$.

We now want to use this identification to describe Legendrian knots
in $V_0$ topologically isotopic to $S^1\times\{0\}$. We call the
Reeb orbit in (i) the \emph{spine} of~$V_{S^3}$.
For the definition of the classical invariants $\tb$ (Thurston--Bennequin
invariant) and $\rot$ (rotation number) of Legendrian knots
see~\cite[Chapter~3]{geig08}.

\begin{lem}
The closed curve
\[ \gamma\co t\longmapsto \frac{1}{\sqrt{2}}\,
(\rme^{\rmi t},\rme^{-\rmi t}),\;\;\; t\in S^1=\R/2\pi\Z,\]
defines an oriented Legendrian knot $L=\gamma(S^1)$ in $(V_{S^3},\xist)$
topologically isotopic in $V_{S^3}$ to the spine of that solid
torus (and hence a topological unknot in~$S^3$). The classical invariants
of this Legendrian unknot, regarded as a knot
in $(S^3,\xist)$, are $\tb(L)=-1$ and $\rot(L)=0$.
\end{lem}

\begin{proof}
We have $\dot{\gamma}=\partial_{\varphi_1}-\partial_{\varphi_2}$,
which gives $\alphast(\dot{\gamma})=0$, so $L$
is a Legendrian knot. It is a $(1,-1)$-curve on the
Hopf torus
\[ \bigl\{|z_1|=|z_2|=\nicefrac{1}{\sqrt{2}}\bigr\}\subset V_{S^3},\]
and hence topologically isotopic to the spine of $V_{S^3}$.

The push-off $L'$ of $L$ in the Reeb direction is a parallel
$(1,-1)$-curve on the Hopf torus, so the linking number
of $L$ with $L'$ in $S^3$, which by definition
is the Thurston--Bennequin invariant of~$L$, equals~$-1$.

From $\tb(L)=-1$ and $L\subset S^3$ being a topological unknot,
one can conclude $\rot(L)=0$  with the help of
of the Bennequin inequality~\cite[Theorem~4.6.36]{geig08}.
Here is a direct proof.
We need to verify that along
$\gamma$ the vector field $\dot{\gamma}$ does not rotate
relative to a global trivialisation of the plane field~$\xist$.
A global trivialisation of the standard contact structure $\xist$
on $S^3$ is given by the vector field
\begin{equation}
\label{eqn:trivial}
y_2\partial_{x_1}+x_2\partial_{y_1}-y_1\partial_{x_2}-x_1\partial_{y_2}.
\end{equation}
Along the Legendrian curve
\[ \gamma(t)=\frac{1}{\sqrt{2}}(\cos t,\sin t,\cos t,-\sin t)\]
this vector field equals
\[ \frac{1}{\sqrt{2}}\bigl(-\sin t\,\partial_{x_1}+\cos t\,\partial_{y_1}
-\sin t\, \partial_{x_2}-\cos t\,\partial_{y_2}\bigr),\]
which coincides with the velocity vector field $\dot{\gamma}(t)$.
This proves $\rot(L)=0$.
\end{proof}

Legendrian unknots in $(S^3,\xist)$ have been
classified by Eliashberg and Fraser~\cite{elfr09},
see also \cite[Section~3.5]{etho01} or \cite[Theorem~5.1]{geon15}.

\begin{thm}[Eliashberg--Fraser]
\label{thm:EF}
Let $L\subset (S^3,\xist)$ be an oriented Legendrian unknot. Then
$\tb(L)=n$ with $n$ a negative integer, and $\rot(L)$ lies in the range
\[ \{ n+1, n+3,\ldots ,-n-3,-n-1\}.\]
Any such pair of invariants $(\tb,\rot)$ is realised, and it
determines $L$ up to Legendrian isotopy.
\end{thm}

These Legendrian unknots are obtained from the one with $\tb=-1$
and $\rot=0$ by stabilisation, which is a local process corresponding
to adding zigzags in the front projection. This yields the following result,
which implies that any Legendrian unknot in $S^3$ can be used,
after the identification of $V_{S^3}\subset S^3$ with
$V_0\subset S^1\times S^2$, for the surgery we have in mind.

\begin{prop}
Any Legendrian unknot in $(S^3,\xist)$ can be realised in $V_{S^3}\subset
S^3$ as a knot topologically isotopic in $V_{S^3}$ to the spine of that
solid torus.\qed
\end{prop}

We write $L_{n,r}\subset V_0$ for the Legendrian knot
corresponding to a Legendrian unknot in $(S^3,\xist)$
(as in this proposition) with invariants
$\tb=n$ and $\rot=r$ in the range allowed by Theorem~\ref{thm:EF}.

We now define $\omega_{n,r}$ as the odd-symplectic form
on $S^3$ (by Lemma~\ref{lem:surgery-top}), obtained by
performing contact surgery in the sense of
\cite{elia90,wein91} along $L_{n,r}\subset (V_0,\ker\alpha_0)$.
Here we assume that the surgered-out solid torus is contained
in~$V_0$. Then $\omega_{n,r}$ is defined to coincide with
$\omega_{S^2}$ away from the surgery region, and with $\rmd\alpha$
for some extension $\alpha$ of $\alpha_0$ over the glued-in solid torus
as a contact form defining the contact structure obtained
by surgery.

Contact surgery can be interpreted as a symplectic handle
attachment, cf.~\cite[Section~6.2]{geig08},
and then gives rise to a symplectic cobordism from the
old to the surgered contact manifold. Here we are only dealing with
an odd-symplectic manifold where the odd-symplectic form
comes from a contact form near the surgery region, so we have to
specify what we mean by `symplectic cobordism' in this situation.

\begin{defn}
Let $(M_i,\omega_i)$, $i=0,1$, be two odd-symplectic manifolds that are
closed and of the same dimension. A \emph{symplectic cobordism}
from $(M_0,\omega_0)$ to $(M_1,\omega_1)$ is a compact symplectic
manifold $(W,\Omega)$ with boundary $\partial W=M_1-M_0$ and such that
$\Omega|_{TM_i}=\omega_i$, $i=0,1$.
\end{defn}

\begin{rem}
This is the notion of symplectic cobordism between odd-symplectic
manifolds as in \cite[Definition~6.1]{civo15}, except that they
only consider topologically trivial cobordisms. As Cieliebak and
Volkov point out, this notion of symplectic cobordism
is not, in general, reflexive (unlike in the following example).
This leads them to
consider a weaker notion of symplectic cobordism. For our purposes,
the definition above is appropriate. In particular, such symplectic
cobordisms can be composed by the neighbourhood theorem
for hypersurfaces in symplectic manifolds, so the cobordism relation becomes
transitive. See also the discussion in~\cite{geze17}.
\end{rem}

Write $D^2_{\rho}\subset\R^2$ for a $2$-disc of radius~$\rho$,
and $S^1_{\rho}=\partial D^2_{\rho}$ for the circle of radius~$\rho$.
For $\rho=1$ we simply write~$D^2$ and $S^1$. Let $\omega_{\R^2}$
be the standard symplectic form on $\R^2$ or any subset of it.
Then the symplectic $4$-manifold
\[ (W_0,\Omega_0):=
\bigl((D^2\setminus\Int(D^2_{1/2}))\times S^2,\omega_{\R^2}\oplus
\omega_{S^2}\bigr)\]
is a symplectic cobordism between two copies of 
$(S^1\times S^2,\omega_{S^2})$.

By attaching a symplectic handle along the outer copy of
$(S^1\times S^2,\omega_{S^2})$, corresponding to the contact
surgery along $L_{n,r}$, we obtain the following result.

\begin{prop}
\label{prop:cobordism}
There is a symplectic cobordism $(W_{n,r},\Omega_{n,r})$
from $(S^1\times S^2,\omega_{S^2})$ to $(S^3,\omega_{n,r})$.\qed
\end{prop}
\section{The $d_3$-invariant of tangent $2$-plane fields}
\label{section:d3}
In order to distinguish the odd-symplectic forms $\omega_{n,r}$
on $S^3$ we consider what is essentially the Hopf invariant
of the oriented line field $\ker\omega_{n,r}$. A normalised
section of this line field,
together with a trivialisation of the tangent bundle $TS^3=S^3\times\R^3$,
determines a map $S^3\rightarrow S^2$, and the Hopf invariant would be the
linking number of the preimages of two regular values.

For our purposes, an invariant introduced by Gompf~\cite{gomp98},
and now commonly called the $d_3$-invariant, is more amenable
to computations. It is based on realising a closed $3$-manifold
with a tangent $2$-plane field $\eta$ as the boundary of a compact
almost complex $4$-manifold such that $\eta$ is the complex line field
along the boundary. Since our manifolds come from a construction
via symplectic cobordisms, we are in a natural setting for
computing the $d_3$-invariant. In our application,
the plane field $\eta$ will be the the one complementary to
$\ker\omega_{n,r}$, with orientation defined by the $2$-form $\omega_{n,r}$.

We first recall the definition of the $d_3$-invariant.
Thus, let $Y$ be a closed, oriented $3$-manifold, and $\eta\subset TY$
an oriented tangent $2$-plane field. Suppose that $Y$ is the
boundary of a compact almost complex $4$-manifold $(X,J)$
such that $\eta$ equals the complex line field $TY\cap J(TY)$.
In fact, using obstruction theory, one can show that such
an $(X,J)$ exists for any given $(Y,\eta)$,
see~\cite[Lemma~4.4]{gomp98}; in our situation, this will
be evident by construction.

When the first Chern class $c_1(\eta)$ is a torsion class (hence
trivially in the case $Y=S^3$), one can make sense of
the number $c_1^2(X,J)$ obtained by squaring the first Chern
class of the almost complex structure $J$ on~$X$, even though
$H^2(X;\Z)$ does not have a well-defined intersection pairing.
We shall discuss this issue in the actual computations below.

Write $\sigma(X)$ for the signature of $X$, that is, the signature
of the intersection pairing on $H_2(X;\Z)$. By $\chi(X)$ we denote
the Euler characteristic of~$X$. The following
Theorem is due to Gompf~\cite[Theorem~4.16]{gomp98}; see also
\cite[Section~11.3]{gost99} and~\cite[Section~2]{dgs04}.

\begin{thm}[Gompf]
For $c_1(\eta)$ a torsion class, the rational number
\[ d_3(\eta)=\frac{1}{4}\bigl(c_1^2(X,J)-3\sigma(X)-2\chi(X)\bigr)\]
is a homotopy invariant of~$\eta$.
\end{thm}

This invariant is complete in the following sense: if two
tangent $2$-plane fields $\eta_0,\eta_1$ are homotopic over the $2$-skeleton
of $Y$, and $c_1(\eta_0)=c_1(\eta_1)$ is a torsion class,
then $\eta_0$ is homotopic to $\eta_1$ if and only if
$d_3(\eta_0)=d_3(\eta_1)$.

\begin{rem}
For $S^3$ the $d_3$-invariant takes values in $\Z+\nicefrac{1}{2}$.
If we take the trivialisation of $TS^3$ coming from regarding
$S^3$ as the unit quaternions, the relation with the Hopf
invariant $h$ is given by $d_3=-h-\nicefrac{1}{2}$. In particular,
we have $d_3(\xist)=-\nicefrac{1}{2}$, which accords with
the formula for $d_3$ when we regard $(S^3,\xist)$ as the boundary
of the unit ball in~$\C^2$.
\end{rem}

\section{Computing the $d_3$-invariant for the examples}
\label{section:d3-examples}
Given an odd-symplectic form $\omega$ on $S^3$, we define
$d_3(\omega)$ as $d_3(\eta)$ for an oriented tangent
$2$-plane field $\eta$ with $\omega|_{\eta}>0$. In order to
find an almost complex filling of $(S^3,\omega_{n,r})$ as
required for the computation of its $d_3$-invariant,
we start from $D^2\times S^2$ with the obvious (almost)
complex structure as a filling of $(S^1\times S^2,\omega_{S^2})$.
Write $(X_{n,r},J_{n,r})$ for the filling of $(S^3,\omega_{n,r})$
given by the symplectic handle attachment. In statements that
do not depend on the specific choice of $n$ and $r$ we simply
write $(X,J)$.
\subsection{The homology of $X$}
The $4$-manifold $X$ is a handlebody obtained from
a $4$-ball by attaching two $2$-handles, the first one
producing $D^2\times S^2$, the second one corresponding to
the contact surgery. In particular, we have $\chi(X)=3$.

The homology of $X$ is generated by the $2$-spheres
\[ S^2_0=\{0\}\times S^2\subset D^2\times S^2\subset X\]
and $S^2_1$, made up of the core disc of the second handle and
a Seifert disc bounded by the Legendrian knot
$L=L_{n,r}$ in $D^2\times S^2$. This
Seifert disc is topologically isotopic to $D^2\times\{*\}$ in $D^2\times S^2$
(via an isotopy keeping the boundary of the disc in $S^1\times S^2$).
 
For the computation of the
signature $\sigma(X)$, we need to understand the intersection numbers
of the generating $2$-spheres. The $2$-sphere $S^2_0$ is
given the orientation induced by $\omega_{S^2}$, i.e.\
the orientation of the $S^2$-factor in $S^1\times S^2$.
The orientation chosen on $S^2_1$ is the one that coincides
with the positive orientation of $D^2\times\{*\}$ over the
Seifert disc (and with $L$ as the positive boundary of
the Seifert disc).
We then clearly have
\[ S^2_0\bullet S^2_0=0\]
and
\[ S^2_0\bullet S^2_1=1.\]
Hence, no matter what the intersection number $S^2_1\bullet S^2_1$,
the intersection matrix will have signature $\sigma(X)=0$. For the
subsequent computations, however, we need to establish the value
of this intersection number.

The linking number of any two orbits of the standard Reeb vector field
$\Rst$ in $S^3\subset\C^2$ equals 1: the two orbits in either complex
coordinate plane make up a positive Hopf link; the other orbits
are $(1,1)$-curves on the Hopf tori. Under the identification of
$V_{S^3}$ with $V_0$, the Reeb orbits are sent to curves in the class
$S^1\times\{*\}$. This means that the $2$-discs in $D^2\times S^2$
bounded by $L$ and its push-off $L'$ transverse to the contact structure
intersect in one point less than the corresponding discs in $S^3$;
the latter intersection number is the linking number of $L$ and
$L'$ in~$S^3$, which is $\tb(L)$ by definition.
The framing for the contact surgery is $-1$ (i.e.\ one left twist)
relative to the contact framing defined by this push-off~$L'$,
see~\cite[Example~6.2.7]{geig08}. In the surgery handle,
$L$ bounds the core disc, and its push-off with the extra negative
twist a $2$-disc parallel to the core disc. It follows that
the self-intersection number of $S^2_1$ is accounted for by
the intersection of the mentioned discs in $D^2\times S^2$, hence
\[ S^2_1\bullet S^2_1=\tb(L_{n,r})-2=n-2,\]
where $\tb$ is computed in~$(S^3,\xist)$.

So the intersection form on $H_2(X;\Z)$ is given by
\begin{equation}
\label{eqn:linking}
Q_n:=\begin{pmatrix}
0 & 1\\
1 & n-2
\end{pmatrix}.
\end{equation}
\subsection{The first Chern class}
We now want to compute $c_1=c_1(X,J)$. To this end, we need to understand
the restriction of $TX$ to $S^2_0$ and $S^2_1$. Along $S^2_0$,
the tangent bundle $TX$ splits as $TS^2_0\oplus\C$, hence
\[ \langle c_1,[S^2_0]\rangle =2.\]

For $S^2_1$, the argument is more involved and analogous
to the proof of \cite[Proposition~2.3]{gomp98}.
The symplectic $2$-handle $D^2\times D^2$ for the contact
surgery (which we attach along its lower boundary $\partial D^2\times D^2$
to $D^2\times S^2$), regarded as an almost complex manifold,
is given by $D^2\times D^2\subset\R^2\times\rmi\R^2$,
see \cite[Figure~6.4]{geig08}. Write $q_1,q_2$ for the cartesian
coordinates of the first $\R^2$-factor.

Along the attaching circle of the $2$-handle, that is,
\[ \partial D^2\times\{0\}\subset \partial D^2\times D^2
\subset D^2\times D^2,\]
we take the complex trivialisation of the tangent bundle
of $D^2\times D^2$ given by the tangent vector $\tau$ to
the attaching circle and the outward normal vector $\nu$.
This frame differs from the product frame $\partial_{q_1},\partial_{q_2}$
by a generator of $\pi_1(\SO(2))$. It follows that $\tau,\nu$ extends
as a complex frame over $D^2\times D^2$, since
\[ \SO(2)\subset \SU(2)\subset \U(2),\]
and $\SU(2)$ is simply connected.

\begin{figure}[h]
\labellist
\small\hair 2pt
\pinlabel $\nu$ [b] at 189 325
\pinlabel $\nu$ [t] at 396 172
\pinlabel ${V_0\subset S^1\times D^2}$ [l] at 433 322
\pinlabel $\text{$2$-handle}$ [bl] at 451 192
\pinlabel $S^2_1$ [bl] at 235 270
\pinlabel $\C^2_{x_2,y_2}$ [r] at 287 496
\pinlabel $\C^2_{x_1,y_1}$ [t] at 560 251
\pinlabel ${D^2\times S^2}$ at 350 436
\endlabellist
\centering
\includegraphics[scale=0.6]{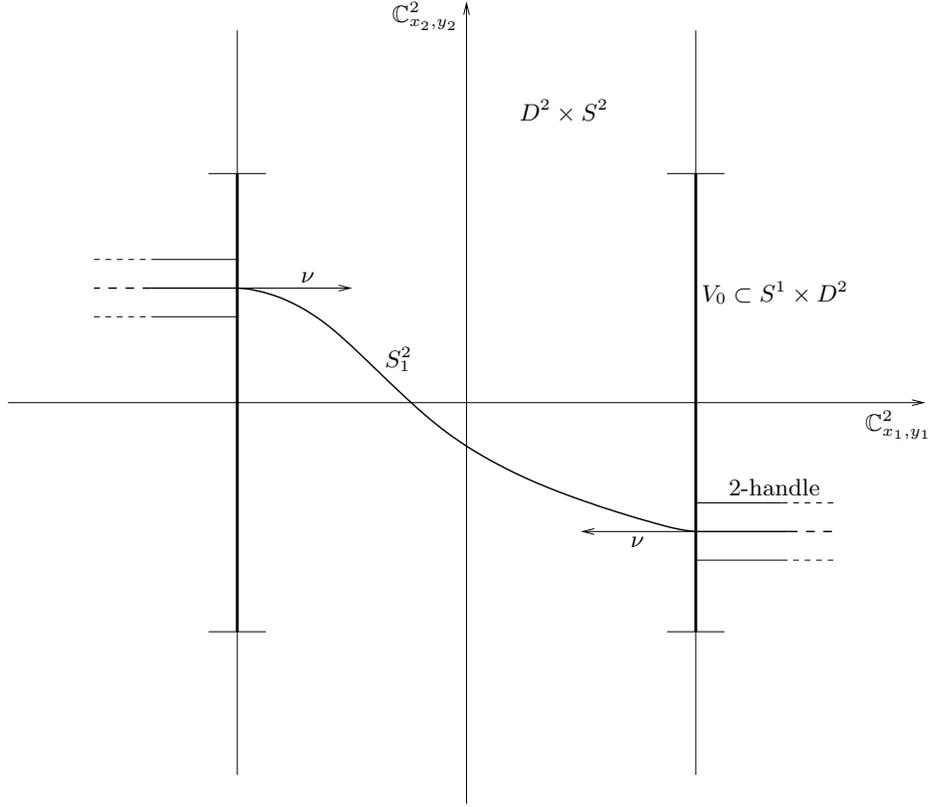}
  \caption{Computing $c_1$ from the handle attachment}
  \label{figure:trivialisations}
\end{figure}

For the following discussion and the notation used
see Figure~\ref{figure:trivialisations}, which shows the
Seifert disc bounded by $L$ in $D^2\times S^2$ as part of the
$2$-sphere $S^2_1$. Recall that $S^2_1$ is given the orientation
which restricts to the positive orientation of that disc.
The second $\C$-factor is regarded as a subset of $S^2$
in $D^2\times S^2$.

The almost complex structure that extends over a symplectic handle
corresponding to a contact surgery is the one coming from the
symplectisation, i.e.\ it is one which preserves the contact planes
and sends the normal direction to the contact manifold to
the Reeb direction (or, up to homotopy, to any direction
in the contact manifold transverse to the contact structure).

The $S^1$-factor in $V_0=S^1\times D^2$ is transverse to
the contact structure that we use to define the surgery, so
the contact planes are homotopic (via planes transverse to the
$S^1$-factor) to planes tangent to the $D^2$-factor.
Hence, up to homotopy, we may assume that the almost complex structure
on $D^2\times S^2\subset X$ is simply the one corresponding to
this product structure, and that the vector field
$\tau$ is tangent to the second factor

\begin{lem}
The vector fields $\nu$ over the $2$-handle and $\partial_{x_1}$
over the Seifert disc of $L$ in $D^2\times S^2$
span a complex line bundle $\LL_1$ over $S^2_1$
with $c_1(\LL_1)=1$.
\end{lem}

\begin{proof}
Along the attaching circle of the $2$-handle, which is identified
with the Legendrian knot $L$, both $\nu$ and $\partial_{x_1}$
lie in the first $\C$-factor, so they span a complex line bundle
$\LL_1$ over all of $S^2_1$. The frame $\partial_{x_1}$ along $L$
extends over the Seifert disc. The vector field $\nu$, which is defined
over the part of $S^2_1$ made up by the core disc
in the handle, makes one positive twist relative to $\partial_{x_1}$
along~$L$. From the interpretation of the first Chern class as a
relative Euler class it follows that $c_1(\LL_1)=1$.
\end{proof}

\begin{lem}
The vector fields $\tau$ over the $2$-handle and $\partial_{x_2}$
over the Seifert disc of $L=L_{n,r}$ span a complex line bundle
$\LL_2$ with $c_1(\LL_2)=\rot(L_{n,r})-2=r-2$.
\end{lem}

\begin{proof}
As discussed above, we may think of $\tau$ along $L$ as a vector field
in the second $\C$-factor, as is~$\partial_{x_2}$. So the two vector
fields define a complex line bundle $\LL_2$ over~$S^2_1$.

Along the spine of $V_{S^3}$, i.e.\ the Reeb orbit $\theta\mapsto
(\rme^{\rmi\theta},0)$, the vector field in (\ref{eqn:trivial}) 
becomes
\[ -\sin\theta\, \partial_{x_2}-\cos\theta\, \partial_{y_2}.\]
Recall that this vector field defines a global trivialisation of $\xist$,
and the rotation number of an oriented Legendrian knot in
$(S^3,\xist)$ counts the rotations of the velocity vector relative to
this frame. From the expression above we see that this frame
makes one left-handed twist relative to the Seifert framing
of the spine given by~$\partial_{x_2}$. In other words, a push-off
of the spine in the direction of the global frame of $\xist$
has linking number~$-1$ in $S^3$ with the spine.

On the other hand, the Reeb orbits of~$\Rst$, which make up the
Hopf fibration, have pairwise linking number~$+1$.
This means that the `Reeb framing' of $\xist$ over $V_{S^3}$
given by parallel Reeb orbits makes two positive twist
relative to the global framing. Hence, the velocity vector field
of a oriented Legendrian knot in $V_{S^3}$ makes $\rot-2$
rotations relative to the Reeb framing.

In $V_0\subset S^1\times S^2$ the Reeb orbits are of the form
$S^1\times\{*\}$, so the Reeb framing corresponds to the vector field
$\partial_{x_2}$. Our discussion shows that $\tau$ makes
$r-2$ twists relative to this frame, which implies
$c_1(\LL_2)=r-2$.
\end{proof}

Thus, over the sphere $S^2_0$ the tangent bundle splits
as $\LL_1\oplus\LL_2$, and the two preceding lemmata imply that
$c_1(X,J)$ satisfies
\[ \langle c_1,[S^2_1]\rangle=r-1.\]
\subsection{Computing $c_1^2$ and $d_3$}
The general procedure for computing $c_1^2(X,J)$ when
$c_1(X,J)$ restricts to a torsion class on the boundary $\partial X$
is explained in \cite[Section~2]{dgs04}. In our situation,
where $\partial X=S^3$, one can give a more direct argument.

The homology long exact sequence of the pair $(X,\partial X=S^3)$,
with integral coefficients understood, gives us an isomorphism
\[ \varphi\co H_2(X)\longrightarrow H_2(X,\partial X).\]
Given a cohomology class $c\in H^2(X)$, there is a unique
homology class $C\in H_2(X)$ with $\varphi (C)=\PD(c)$,
where $\PD\co H^2(X)\rightarrow H_2(X,\partial X)$ is the
Poincar\'e duality isomorphism. Then $c^2$ is defined as
the intersection number $C^2:=C\bullet C\in\Z$.

As mentioned before, $H_2(X)$ is generated by the homology
classes $[S^2_0]$ and $[S^2_1]$. The relative homology group
$H_2(X,\partial X)$ is generated by the classes $[N_0],[N_1]$
of normal discs $N_0,N_1$ to the knots along which the
$2$-handles are attached to a $4$-ball~$D^4$. Here the handle which gives
rise to $N_0$ is the one attached along a trivial knot in $D^4$
to produce $D^2\times S^2$, that is, we may take $N_0=D^2\times\{*\}
\subset D^2\times S^2$, and we choose this disc with
$\partial D^2\times\{*\}$ disjoint from $V_0$, where the second
handle is attached. The orientations of $N_0,N_1$ are chosen such that
\[ N_0\bullet S^2_0=1=N_1\bullet S^2_1.\]
In terms of the bases $[S^2_0],[S^2_1]$ and $[N_0],[N_1]$
for $H_2(X)$ and $H_2(X,\partial X)$, respectively, the
isomorphism $\varphi$ is given by the linking matrix $Q_n$
shown in~(\ref{eqn:linking}).

The Poincar\'e dual of $c_1$ satisfies
\[ c_1(S^2_i):=\langle c_1,[S^2_i]\rangle=\PD(c_1)\bullet [S^2_i],\]
which implies that
\[ \PD(c_1)=c_1(S^2_0)\cdot [N_0]+c_1(S^2_1)\cdot [N_1]=
2[N_0]+(r-1)[N_1].\]
Hence, the class $C_1=a_0[S^2_0]+a_1[S^2_1]$ with $\varphi(C_1)=\PD(c_1)$
is found by solving
\[ \begin{pmatrix}0&1\\ 1&n-2\end{pmatrix}
\begin{pmatrix}a_0\\a_1\end{pmatrix}=
\begin{pmatrix}2\\r-1\end{pmatrix}.\]
This gives $a_0=r-2n+3$ and $a_1=2$. It follows that
\[ c_1^2=C_1^2=
(a_0,a_1)
\begin{pmatrix}0&1\\ 1&n-2\end{pmatrix}
\begin{pmatrix}a_0\\a_1\end{pmatrix}=
(r-2n+3,2)\begin{pmatrix}2\\r-1\end{pmatrix}=4r-4n+4\]
and
\begin{eqnarray*}
d_3(\omega_{n,r}) & = & \frac{1}{4}\bigl(c_1^2(X_{n,r},J_{n,r})-
                          3\sigma(X_{n,r})-2\chi(X_{n,r})\bigr)\\
  & = & r-n-\frac{1}{2}.
\end{eqnarray*}
With Theorem~\ref{thm:EF} we conclude that the values
\[ \nicefrac{1}{2},\;\nicefrac{5}{2},\;\nicefrac{9}{2},\ldots\]
can be realised as $d_3$-invariant. This proves the part of
Theorem~\ref{thm:main} saying that we can obtain
infinitely many distinct odd-symplectic forms on $S^3$
via our surgery construction.

\begin{rem}
Given an odd-symplectic manifold $(M^{2n-1},\omega)$ and a $1$-form
$\beta$ on $M$ with $\beta\wedge\omega^{n-1}>0$, one can form
the \emph{symplectisation} $\bigl((-\varepsilon,\varepsilon)\times M,
\omega+\rmd(t\beta)\bigl)$ for $\varepsilon>0$
sufficiently small. Since $(S^3,\omega_{n,r})$ is distinguished
homotopically from $(S^3,\omegast|_{TS^3})$, there can be no
symplectic embedding of a symplectisation of the former into
one of the latter, and no homotopically nontrivial
symplectic embedding in the other direction. Any such embedding
would give rise to a topologically trivial
symplectic cobordism between the two odd-symplectic manifolds, and hence
to a homotopy of odd-symplectic forms.
\end{rem}
\section{Proof of Theorem~\ref{thm:main}}
\label{section:proof}
Let $\omega=\omega_{n,r}$ be one of the odd-symplectic forms on $S^3$
coming from our construction, and let $(W,\Omega)=(W_{n,r},\Omega_{n,r})$
be the corresponding cobordism from $(S^1\times S^2,\omega_{S^2})$
to $(S^3,\omega)$ as in Proposition~\ref{prop:cobordism}.

By \cite[Lemma~5]{geze17}, there is a symplectic cobordism
from $(-S^1\times S^2,\omega_{S^2})$ to $(S^1\times S^2,\omega_{S^2})$,
which amounts to a cobordism from the empty set to two copies
of $(S^1\times S^2,\omega_{S^2})$. By gluing a copy
of $(W,\Omega)$ along each boundary component we obtain a
compact symplectic manifold with two boundary components
$(S^3,\omega)$.

Now, assuming there were a symplectic cobordism as in
Question~\ref{quest:Fish} from $(S^3,\omega)$ to $(S^3,\xist)$,
we could glue a copy of this cobordism to each boundary of the
previously constructed symplectic manifold. The result would be
a weak symplectic filling of a disjoint union $(S^3,\xist)\sqcup
(S^3,\xist)$, which cannot exist by~\cite[Theorem~1.4]{mcdu91}.

This contradiction proves that the answer to Fish's question is negative
for the odd-symplectic forms~$\omega_{n,r}$ on~$S^3$,
which completes the proof of Theorem~\ref{thm:main}.
\section{On the dynamics of the examples}
\label{section:dynamics}
Outside the surgery region, the odd-symplectic forms $\omega_{n,r}$
coincide with $\omega_{S^2}$, whose characteristic direction
is given by~$\partial_{\theta}$. In particular,
on $S^1\times S^2\setminus V_0$ (which we may regard as a subset
of the surgered~$S^3$), the characteristic flow of $\omega_{n,r}$
is periodic and certainly not minimal.

We now modify $\omega=\omega_{n,r}$ into an odd-symplectic form
whose non-minimality is no longer apparent, but which still
does not admit a cobordism as requested by Fish. For these examples, then,
minimality is an open question.

Write $V\subset V_0$ for the solid torus containing
$L_{n,r}$ that we cut out when we perform the
surgery (equivalently, the image of the attaching map
of the $2$-handle). Write $\pi\co S^1\times S^2\rightarrow S^2$
for the projection onto the $S^2$-factor. Define $A:=\pi^{-1}(\pi(V))$;
by $\calO(A)\subset V_0$ we denote a slight thickening of
the $S^1$-invariant set~$A$.
Let $\psi\co S^1\times S^2\rightarrow\R$ be a smooth function,
independent of the $S^1$-coordinate, and let $X_{\psi}$
be the $S^1$-invariant vector field which on each slice $\{*\}\times S^2$
equals the Hamiltonian vector field of~$\psi$, that is,
\[ \omega_{S^2}(X_{\psi},\,.\,)=-\rmd\psi.\]
We may choose $\psi$ in such a way that $\rmd\psi$ does not
have any zeros outside~$A$, and such that
each level set of $\psi$ (and hence each flow line of~$X_{\psi}$) meets~$A$.

On $S^1\times S^2$ we define a $1$-form $\beta_0$ by setting it
equal to the contact form $\alpha_0$ on $A$, equal to
$\psi\,\rmd\theta$ on $S^1\times S^2\setminus\calO(A)$,
and as a convex interpolation on the collar region in between.
Then $\beta_0\wedge\omega_{S^2}>0$ globally on $S^1\times S^2$.
After the surgery, where $\alpha_0$ is extended over the
glued-in solid torus as a contact form~$\alpha$, this extension
defines a $1$-form $\beta$ with $\beta\wedge\omega$ globally on~$S^3$,
and still $\beta=\psi\,\rmd\theta$ on $S^1\times S^2\setminus\calO(A)$,
now regarded as a subset of~$S^3$.

The $2$-form
\[ \Omega:=\omega+\rmd(t\beta)=\omega+\rmd t\wedge\beta+t\,\rmd\beta\]
is a symplectic form on $[0,\varepsilon]\times S^3$ for
$\varepsilon>0$ sufficiently small, and it pulls back to
the odd-symplectic forms $\omega$ on $S^3\times\{0\}$ and
to $\omega_{\varepsilon}:=\omega+\varepsilon\,\rmd\beta$
on $S^3\times\{\varepsilon\}$, respectively.
The odd-symplectic manifold
$(S^3,\omega_{\varepsilon})$ does not admit a cobordism
as in Question~\ref{quest:Fish}, for otherwise
we would obtain one for $(S^3,\omega)$ by concatenation.

On $S^1\times S^2\setminus\calO(A)\subset S^3$ we have
\[ \omega_{\varepsilon}=\omega+\varepsilon\,\rmd(\psi\,\rmd\theta)=
\omega+\varepsilon\,\rmd\psi\wedge\rmd\theta.\]
The characteristic direction of this odd-symplectic form is
given by $\partial_{\theta}-\varepsilon X_{\psi}$. By our choice of~$\psi$,
this means all characteristics intersect $\calO(A)$, and their global
dynamics can no longer be controlled.

\begin{ack}
We are grateful to Joel Fish for suggesting this problem about the
existence of symplectic cobordisms. We also thank Alberto Abbondandolo,
Peter Albers and Ana Rechtman for useful conversations, and Marc Kegel
for his critical reading of a draft version of this paper and several
constructive comments.
\end{ack}

\end{document}